\documentclass[10pt,twoside,english,reqno,a4paper]{amsart}

\usepackage{listings,graphicx,amsmath,varioref,amscd,amssymb,color,bm,stmaryrd,amsthm,amsfonts,graphics,geometry,latexsym,pgf,pst-all} 
\theoremstyle{plain}
\usepackage{esint}
\usepackage{amsthm}

\usepackage{enumerate}

\usepackage{caption}
\usepackage{subcaption}
\theoremstyle{plain}
\newtheorem{theorem}{Theorem}[section]
\newtheorem{proposition}[theorem]{Proposition}

\newtheorem{lemma}[theorem]{Lemma}

\newtheorem{corollary}[theorem]{Corollary}

\usepackage{geometry}
\usepackage[colorlinks=false]{hyperref}

\theoremstyle{definition}
\newtheorem{definition}[theorem]{Definition}

\newtheorem{remark}[theorem]{Remark}

\theoremstyle{remark}

\usepackage{fouriernc}
\usepackage[T1]{fontenc}

\numberwithin{equation}{section}

\newcommand{\na}{\mathbb{N}}
\newcommand{\re}{\mathbb{R}}
\newcommand{\N}{\mathbb{N}}
\newcommand{\R}{\mathbb{R}}

\def\huz{H^1_0 (\Omega)}
\def\hu{H^1 (\Omega)}

\def\luo{L^{1}(\Omega)}
\def\lio{{L^{\infty}(\Omega)}}

\def\co{{C(\overline\Omega)}}

\def\into{\int_{\Omega}}
\def\intdo{\int_{\partial\Omega}}
\def\ae{\mathrm{a.e.}}

\newcommand{\ou}{\overline{u}}
\newcommand{\tu}{\tilde{u}}

\begin{document}
\title[Multiplicity of solutions for semilinear Robin problems involving sign-changing nonlinearities]{Multiplicity of nonnegative solutions for semilinear Robin problems involving sign-changing nonlinearities}

\author[J. Carmona Tapia]{Jos\'{e} Carmona Tapia}
\address[Jos\'{e} Carmona Tapia]{Departamento de Matem\'aticas\\ Uni\-ver\-si\-dad de Alme\-r\'ia\\Ctra. Sacramento s/n\\
La Ca\-\~{n}a\-da de San Urbano\\ 04120 - Al\-me\-r\'{\i}a, Spain}
\email{jcarmona@ual.es}

\author[A. J. Mart\'{i}nez Aparicio]{Antonio J. Mart\'{i}nez Aparicio}
\address[Antonio J. Mart\'{i}nez Aparicio]{Departamento de Matem\'aticas\\ Uni\-ver\-si\-dad de Alme\-r\'ia\\Ctra. Sacramento s/n\\
La Ca\~{n}ada de San Urbano\\ 04120 - Al\-me\-r\'{\i}a, Spain}
\email{ajmaparicio@ual.es}

\author[P. J. Mart\'{i}nez-Aparicio]{Pedro J. Mart\'{i}nez-Aparicio}
\address[Pedro J. Mart\'{i}nez-Aparicio]{Departamento de Matem\'aticas\\ Uni\-ver\-si\-dad de Alme\-r\'ia\\Ctra. Sacramento s/n\\
La Ca\~{n}ada de San Urbano\\ 04120 - Al\-me\-r\'{\i}a, Spain}
\email{pedroj.ma@ual.es}

\keywords{Robin boundary conditions, Positive solutions, Nonlinearities having zeros} \subjclass[2020]{35B09, 35B40, 35J61}

\begin{abstract}
In this article, we investigate the existence and multiplicity of solutions to the Robin problem
\begin{equation*}
	\begin{cases}
		-\Delta u = \lambda f(u) & \text{in } \Omega,\\
		\frac{\partial u}{\partial \nu} + \gamma u=0 & \text{on } \partial\Omega,
	\end{cases}
\end{equation*}
where $\Omega\subset \re^N$ ($N\geq 1$) is a smooth bounded domain, and $\lambda, \gamma>0$. Our main assumption is that $f\colon \re\to \re$ is a locally Lipschitz function, possibly sign-changing, such that $f(s)>0$ for every $s\in (\alpha,\beta)$, where $0<\alpha<\beta$ are two zeros of $f$. Without any further conditions, we establish the existence of two nonnegative solutions whose maximum lies in $(\alpha,\beta)$ for sufficiently large $\lambda$. 

Moreover, we analyse the limiting behaviour of the solution set of this Robin problem, showing that it degenerates into that of the associated Neumann problem as $\gamma\to 0$ and into that of the associated Dirichlet problem as $\gamma\to\infty$.
\end{abstract}

\maketitle
 
\tableofcontents

\section{Introduction}

We study the following Robin problem
\begin{equation}
    \tag{$P_{\lambda,\gamma}$}
	\label{eq:PbLambda}
	\begin{cases}
		-\Delta u = \lambda f(u) & \text{in } \Omega,\\
		\frac{\partial u}{\partial \nu} + \gamma u=0 & \text{on } \partial\Omega,
	\end{cases}
\end{equation}
where $\Omega\subset \re^N$ ($N\geq 1$)  is a bounded domain with $C^2$ boundary, $\nu(x)$ denotes the outward unit normal vector to $x\in \partial\Omega$, and $\lambda,\gamma > 0$ are two parameters. Here, $f\colon \R\to \R$ is a locally Lipschitz function that has at least two positive zeros, $0<\alpha<\beta$, such that $f(s)>0$ for all $s\in (\alpha,\beta)$. We stress that no additional assumptions are made on $f$; in particular, $f$ is allowed to change sign. Note that $\gamma=0$ corresponds to the Neumann boundary condition, while $\gamma=\infty$ formally leads to the Dirichlet boundary condition. 

Our main objective is to study the set of the nonnegative solutions to~\eqref{eq:PbLambda} whose maximum lies between $\alpha$ and $\beta$, i.e., the solutions of~\eqref{eq:PbLambda} that belong to
\[
\mathcal{O}_{\alpha\beta} = \{u\in\co : u\geq 0 \text{ and } \alpha<\|u\|_\co < \beta\}.
\]
In addition to existence and multiplicity, we also investigate the asymptotic behaviour of the solution set of~\eqref{eq:PbLambda} as $\gamma$ tends to $0$ and to $\infty$. Specifically, we show that it converges to the solution set of the corresponding Neumann problem as $\gamma\to 0$, and to that of the associated Dirichlet problem as $\gamma\to\infty$. For this analysis, we demonstrate the following stability result: when $\gamma\to\infty$, solutions of~\eqref{eq:PbLambda} converge strongly in $C(\overline \Omega)$ to solutions of the Dirichlet problem~\eqref{eq:PbDirich}. As far as we know, this type of convergence, key in our study, has not been previously established in the literature.

The existence of solutions belonging to $\mathcal{O}_{\alpha\beta}$ has been extensively studied when the boundary conditions are Dirichlet. For the Dirichlet problem
\begin{equation}
	\label{eq:PbDirich}
	\begin{cases}
		-\Delta u = \lambda f(u) & \text{in } \Omega,\\
		u=0 & \text{on } \partial\Omega,
	\end{cases}
\end{equation}
the strong maximum principle precludes the existence of a solution $u\geq 0$ with $\|u\|_{\co}=\sigma$, where $\sigma>0$ is such that $f(\sigma)\leq 0$. In particular, no nonnegative solution to~\eqref{eq:PbDirich} can have a maximum value of $\alpha$ or $\beta$. In~\cite{Hess}, assuming $f(0)>0$, the author (inspired by~\cite{Brown-Budin}) showed that an area condition on $f$ is sufficient to ensure the existence of two positive solutions of~\eqref{eq:PbDirich} with $\co$-norm between $\alpha$ and $\beta$ for any $\lambda$ greater than some $\overline\lambda>0$. Specifically, this area condition is given by
\begin{equation}
    \label{eq:hyp_hess}
    \int_s^\beta f(t) \ \mathrm{d}t > 0 \text{ for all } s\in[0,\beta).
\end{equation}
Years later, the condition $f(0)>0$ was dropped in~\cite{Cl-Sw}, where the authors showed that~\eqref{eq:hyp_hess} is sufficient to guarantee the existence of a positive solution to~\eqref{eq:PbDirich} belonging to $\mathcal{O}_{\alpha\beta}$ for large $\lambda$. Furthermore, the authors of~\cite{Cl-Sw} and~\cite{Da-Sch} independently proved that condition~\eqref{eq:hyp_hess} is also necessary for the existence of positive solutions to~\eqref{eq:PbDirich} in $\mathcal{O}_{\alpha\beta}$. Consequently, if $r_{\alpha}\in [\alpha,\beta)$ is defined as
\[
r_\alpha:= \inf \left\{ r\in (\alpha, \beta): \int_s^r f(t) \ \mathrm{d}t > 0 \text{ for all } s\in[0,r) \right\},
\]
then any positive solution $u$ to~\eqref{eq:PbDirich} with $\|u\|_\co\in (\alpha,\beta)$ must verify $\|u\|_\co\geq r_{\alpha}$ (see~\cite{Da-Sch}).

Since these pioneering works, such results have been extended to other operators, always assuming Dirichlet boundary conditions. In general, the area condition~\eqref{eq:hyp_hess} is assumed to establish the existence (and sometimes the multiplicity) of solutions in $\mathcal{O}_{\alpha\beta}$, whereas its necessity has been less explored. To mention a few works, this has been done for local operators, including the $p$-Laplacian (\cite{Guo-Webb, Ng-Sch}), the $\phi$-Laplacian (\cite{Correa}), the $p(x)$-Laplacian (\cite{Ho-Kim-Sim}), the $1$-Laplacian (\cite{dS-Fig-Pi}), and a Schr\"{o}dinger type operator (\cite{dS-Sil}), as well as for nonlocal operators such as the fractional Laplacian (\cite{CR}), and a Kirchhoff operator (\cite{A-C-MA}).

Another extensively studied question is the asymptotic behaviour of solutions to~\eqref{eq:PbDirich} that belong to $\mathcal{O}_{\alpha\beta}$ when $\lambda$ diverges. Whenever~\eqref{eq:hyp_hess} holds, it is well-known that the maximal solution $\ou_\lambda$ to~\eqref{eq:PbDirich} in $[0,\beta]$ exists for large $\lambda$, belongs to $\mathcal{O}_{\alpha\beta}$, and converges uniformly to $\beta$ in compact subsets of $\Omega$ (see~\cite{Cl-Sw}). However, determining the behaviour of the second solution $\tilde u_\lambda$ in $\mathcal{O}_{\alpha\beta}$ to~\eqref{eq:PbDirich} is significantly more challenging. As far as we know, such results are available only when $f$ is nonnegative. In this case, $\tilde u_\lambda$ converges to $\alpha$ uniformly in compact subsets of $\Omega$, provided that $f$ is subcritical in a right neighbourhood of $\alpha$ (see~\cite{B-GM-I}). 

More properties of~\eqref{eq:PbDirich} have been explored, for instance, in~\cite{C-MA-MA}, where several nonlinearities with an infinite number of zeros are considered, leading in some cases to the existence of an infinite number of bifurcation points. We also highlight~\cite{Guo}, where the author studies the flat core solutions to~\eqref{eq:PbDirich} that arise when $f$ is not Lipschitz in $\beta$.

For the Robin problem~\eqref{eq:PbLambda}, a natural first question is whether~\eqref{eq:hyp_hess} is necessary and sufficient for the existence of solutions in $\mathcal{O}_{\alpha\beta}$. This question is not only interesting from a purely mathematical point of view, but also for its applications in population dynamics (see~\cite{Shivaji} for further details). We show that the answer to this question is negative: without imposing any assumptions on $f$, we prove that solutions to~\eqref{eq:PbLambda} in $\mathcal{O}_{\alpha\beta}$ always exist for any $\gamma>0$ whenever $\lambda$ is large. To the best of our knowledge, for problems with nonlinearities that have zeros, this is the first existence result of solutions in $\mathcal{O}_{\alpha\beta}$  where~\eqref{eq:hyp_hess} is not required.

Our first result is stated as follows.

\begin{theorem}
\label{th:exist_gam_fix}
Let $f\colon \R \to \R$ be a locally Lipschitz function such that $f(s)>0$ when $s\in (\alpha,\beta)$, where $0<\alpha<\beta$ denote two consecutive zeros of $f$. Then, for every fixed $\gamma>0$, there exists ${\lambda}_{\mathrm{mult}}(\gamma)>0$ such that problem~\eqref{eq:PbLambda} has two solutions in $\mathcal{O}_{\alpha\beta}$ for every $\lambda > {\lambda}_{\mathrm{mult}}$, one of them maximal in $[0,\beta]$.

Furthermore, the maximal solution in $[0,\beta]$, denoted by $\ou_\lambda$, verifies that $\ou_\lambda \to \beta$ in $C(\overline\Omega)$ as $\lambda\to\infty$.

Finally, ${\lambda}_{\mathrm{mult}}(\gamma)\to 0$ as $\gamma\to 0$.
\end{theorem}

Next, we show that there is some continuity between the results we obtain for~\eqref{eq:PbLambda} and the well-known results for the Dirichlet problem~\eqref{eq:PbDirich}. With this aim, whenever~\eqref{eq:hyp_hess} holds, we define
\begin{equation}
    \label{eq:def_lambda_infty}
    \lambda_\infty:= \inf \left\{\lambda\geq 0: \textrm{\eqref{eq:PbDirich} admits solution $u\geq 0$ with } \alpha < \|u\|_{C(\overline\Omega)}<\beta \right\}.
\end{equation}
Due to the results of~\cite{Cl-Sw}, this infimum is both finite and positive. Recall that if~\eqref{eq:hyp_hess} does not hold, then the set over which the infimum is taken is empty (\cite{Cl-Sw}).

As a consequence of Theorem~\ref{th:exist_gam_fix}, we obtain the following result. A key ingredient in the proof is the strong convergence in $\co$ of solutions of~\eqref{eq:PbLambda} to solutions of~\eqref{eq:PbDirich} as $\gamma\to \infty$. To establish this stability result (see Lemma~\ref{lem:stability}), we employ a recent approach by~\cite{Arc-Rez-Sil} based on the Stampacchia regularity method (\cite{Stamp}). To our knowledge, this is the first general result that connects Robin and Dirichlet solutions through strong convergence in $\co$.

\begin{theorem}
\label{th:lambda_bh}
Let $f\colon \R \to \R$ be a locally Lipschitz function such that $f(s)>0$ when $s\in (\alpha,\beta)$, where $0<\alpha<\beta$ denote two consecutive zeros of $f$. Then 
\[
\lambda_{\mathrm{min}}(\gamma):= \inf \left\{\lambda\geq 0: \textrm{\eqref{eq:PbLambda} admits solution $u\geq 0$ with } \alpha < \|u\|_{C(\overline\Omega)}<\beta \right\}
\]
is finite and positive for every $\gamma>0$. Moreover, $\lambda_{\mathrm{min}}(\gamma) \to 0$ as $\gamma\to 0$ and
\begin{enumerate}
    \item[i)] if~\eqref{eq:hyp_hess} holds, then $\lambda_{\mathrm{min}}(\gamma) \to \lambda_\infty$ as $\gamma\to \infty$, where $\lambda_\infty$ is defined in~\eqref{eq:def_lambda_infty},
    \item[ii)] if~\eqref{eq:hyp_hess} does not hold, then $\lambda_{\mathrm{min}}(\gamma) \to \infty$ as $\gamma\to \infty$.
\end{enumerate}
\end{theorem}

According to Theorem~\ref{th:exist_gam_fix}, for each $\lambda>0$ fixed, there exists some $\tilde \gamma (\lambda)>0$ small such that, for any $\gamma\in (0, \tilde \gamma)$, problem~\eqref{eq:PbLambda} has two solutions in $\mathcal{O}_{\alpha\beta}$. To conclude our study, we analyse the behaviour of this pair of solutions when $\gamma$ tends to 0. As $\lambda$ does not play any role, to better illustrate the ideas we fix $\lambda=1$, i.e., we consider problem
\begin{equation}
    \tag{$P_{\gamma}$}
	\label{eq:PbGamma}
	\begin{cases}
		-\Delta u = f(u) & \text{in } \Omega,\\
		\frac{\partial u}{\partial \nu} + \gamma u=0 & \text{on } \partial\Omega.
	\end{cases}
\end{equation}

For small $\gamma>0$, we denote by $\ou_\gamma$ and $\tilde u_\gamma$ the two solutions of~\eqref{eq:PbGamma} in $\mathcal{O}_{\alpha\beta}$ given by Theorem~\ref{th:exist_gam_fix}, where $\overline u_\gamma$ represents the maximal solution in $[0,\beta]$. As we show (see Lemma~\ref{lem:stability}), their limits when $\gamma\to 0$ are always solutions to the Neumann problem
\begin{equation}
\label{eq:PbNeumann}
	\begin{cases}
		-\Delta u = f(u) & \text{in } \Omega,\\
		\frac{\partial u}{\partial \nu}=0 & \text{on } \partial\Omega,
	\end{cases}
\end{equation}
with maximum belonging to the interval $[\alpha,\beta]$. Note that any zero of $f$ is a constant solution to~\eqref{eq:PbNeumann}; in particular, both $\alpha$ and $\beta$ are constant solutions. Our aim is to prove that $\ou_\gamma$ converges to $\beta$ in $\co$ and $\tilde u_\gamma$ converges to $\alpha$ in $\co$ as $\gamma\to 0$. For $\ou_\gamma$, this can be proven without additional assumptions, whereas for $\tilde u_\gamma$, we have to assume the nonexistence of non-constant solutions to~\eqref{eq:PbNeumann}.

When $f\geq 0$, any solution $u$ to~\eqref{eq:PbNeumann} must be constant due to the compatibility condition ${\into f(u)=0}$. However, when $f$ changes sign, the nonexistence of non-constant solutions to~\eqref{eq:PbNeumann} (also known as patterns) is a challenging question that, so far, has only partial answers. Further insights on this issue can be found in~\cite{Ci-Co-Ron} or in~\cite{Nord}, where, assuming the convexity of the domain, the nonexistence of positive patterns to~\eqref{eq:PbNeumann} is addressed. In~\cite{Ci-Co-Ron}, this is proven under the assumption that $f(t)t^{-\frac{N+2}{N-2}}$ is decreasing for $t>0$, while in~\cite{Nord} the same is established assuming that $\sup f'$ is smaller than a constant depending on the domain.

Once the nonexistence of patterns of~\eqref{eq:PbNeumann} is assumed, only two possible limits remain for $\tilde u_\gamma$: either $\alpha$ or $\beta$. Our strategy is to prove that, for small $\gamma>0$, the only solution of problem~\eqref{eq:PbGamma} contained in a left neighbourhood of $\beta$ is $\ou_\gamma$. To show this, we need to add some monotonicity to $f$ near $\beta$; specifically, we assume that there is some $\delta>0$ such that
\begin{equation}
    \label{eq:hyp_mon}
    f'(s)\leq 0 \text{ in } (\beta-\delta, \beta).
\end{equation}

Now, we are ready to state our last result.

\begin{theorem}
\label{th:limit_gamma}
Let $f\colon \R \to \R$ be a locally Lipschitz function such that $f(s)>0$ when $s\in (\alpha,\beta)$, where $0<\alpha<\beta$ denote two consecutive zeros of $f$. For small $\gamma>0$, let $\ou_\gamma$ and $\tilde u_\gamma$ be the two solutions to~\eqref{eq:PbGamma} given by Theorem~\ref{th:exist_gam_fix}, where $\ou_\gamma$ is maximal in $[0,\beta]$. Then,
\[
\ou_\gamma\to \beta \text{ in } \co \text{ as } \gamma\to 0.
\]
Furthermore, if~\eqref{eq:PbNeumann} has no non-constant solutions with maximum in $[\alpha,\beta]$, and if~\eqref{eq:hyp_mon} holds, then 
\[
\tilde u_\gamma\to \alpha \text{ in } \co \text{ as } \gamma\to 0.
\]
\end{theorem}

In Figure~\ref{fig:diagramas}, we combine all our results to sketch the bifurcation diagram of~\eqref{eq:PbLambda} depending on the shape of $f$. In a sense, Robin solutions act as an ``interpolation'' between Neumann and Dirichlet solutions, when these exist. As $\gamma\to 0$, the solution set of~\eqref{eq:PbLambda} degenerates into that of the associated Neumann problem. On the other hand, if a solution to~\eqref{eq:PbDirich} in $\mathcal{O}_{\alpha\beta}$ exists (i.e., if~\eqref{eq:hyp_hess} holds), then as $\gamma\to \infty$, the set of solutions to the Robin problem~\eqref{eq:PbLambda} approaches that of the Dirichlet problem~\eqref{eq:PbDirich}.

\begin{figure}[ht]
\centering
\begin{subfigure}{.09\textwidth}
  \captionsetup{justification=centering}
  \includegraphics[width=1.8cm]{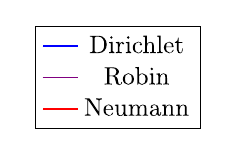}
  \vspace{5mm}
\end{subfigure}
\begin{subfigure}[t]{.3\textwidth}
  \captionsetup{justification=centering}
  \includegraphics[width=4.4cm]{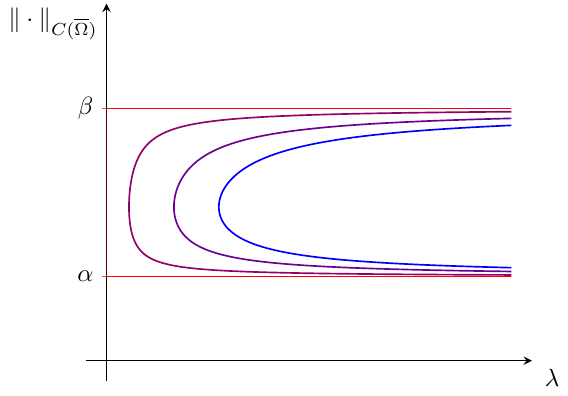}
  \caption{$f\geq 0$}
\end{subfigure}%
\begin{subfigure}[t]{.3\textwidth}
  \captionsetup{justification=centering}
  \includegraphics[width=4.4cm]{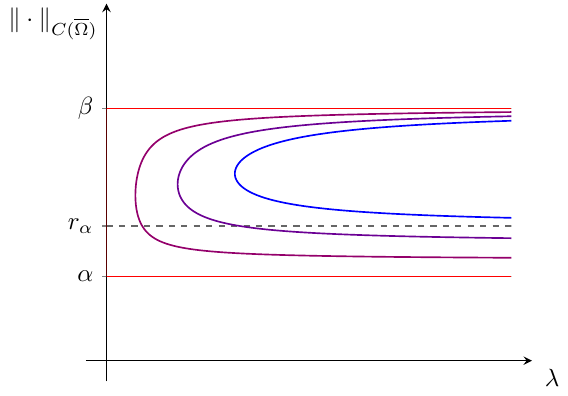}
  \caption{$f$ changes sign and verifies~\eqref{eq:hyp_hess}}
\end{subfigure}%
\begin{subfigure}[t]{.3\textwidth}
  \captionsetup{justification=centering}
  \includegraphics[width=4.4cm]{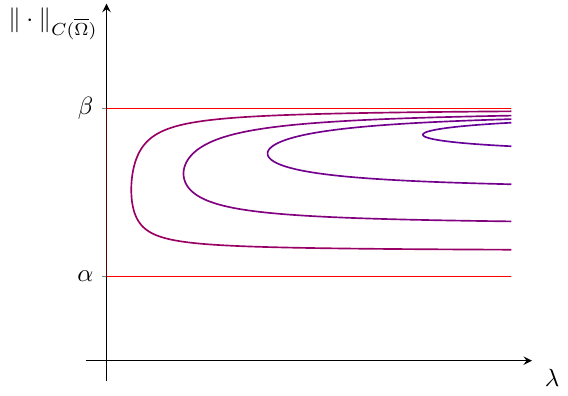}
  \caption{$f$ changes sign and does\\
  not verify~\eqref{eq:hyp_hess}}
\end{subfigure}
\caption{\label{fig:diagramas} Sketch of the bifurcation diagram of problem~\eqref{eq:PbLambda} in $\lambda$, depicted in different shades of purple for various values of $\gamma$, according to Theorem~\ref{th:exist_gam_fix}, Theorem~\ref{th:lambda_bh} and Theorem~\ref{th:limit_gamma}. The solution set of the associated Neumann problem is highlighted in red, while, as long as it exists, the solution set of the Dirichlet problem~\eqref{eq:PbDirich} is shown in blue.
}
\end{figure}

The structure of the article is the following. In Section~\ref{sec:prelim}, we introduce the fundamental tools used throughout this work, including some consequences of the strong maximum principle and a stability lemma. Section~\ref{sec:aux} is devoted to the study of an auxiliary problem resulting from extending $f$ by zero outside $[\alpha,\beta]$. Finally, in Section~\ref{sec:main}, we prove the main results stated in the introduction.

\section{Preliminaries}
\label{sec:prelim}

Throughout the article, we deal with classical solutions. For clarity, we give its definition.

\begin{definition}
    A function $u\in C^2(\overline\Omega)$ is a \textit{subsolution} (resp. a \textit{supersolution}) to~\eqref{eq:PbLambda} if, pointwise, it satisfies
    \begin{equation*}
	\begin{cases}
		-\Delta u \overset{(\geq)}{\leq}\lambda f(u) & \text{in } \Omega,\\
		\frac{\partial u}{\partial \nu} + \gamma u \overset{(\geq)}{\leq} 0 & \text{on } \partial\Omega.
	\end{cases}
\end{equation*}
We say that $u\in C^2(\overline\Omega)$ is a \textit{solution} to~\eqref{eq:PbLambda} if it is both a sub- and a supersolution.
\end{definition}

Sometimes, because of the tools used, we need to consider weak solutions. However, in our setting, classical and weak solutions coincide, as explained in Remark~\ref{rem:weak}. In the following, we recall its definition.

\begin{definition}
    A function $u\in H^1(\Omega)$ is a \textit{weak subsolution} (resp. a \textit{weak supersolution}) to~\eqref{eq:PbLambda} if it verifies
    \begin{equation*}
    \into \nabla u \nabla \varphi + \gamma \intdo u \varphi\  \overset{(\geq)}{\leq} \lambda\into f(u)\varphi,\ \forall \varphi\in \hu \text{ with } \varphi\geq 0.
    \end{equation*}
    In the same way, we say that $u\in H^1(\Omega)$ is a \textit{weak solution} to~\eqref{eq:PbLambda} if
    \begin{equation}
        \label{eq:weak_form}
        \into \nabla u \nabla \varphi + \gamma \intdo u \varphi = \lambda\into f(u)\varphi,\ \forall \varphi\in \hu.
    \end{equation}
\end{definition}

\begin{remark}
\label{rem:weak}
Any bounded weak solution of~\eqref{eq:PbLambda} is actually a classical solution. Indeed, let $u\in \hu\cap \lio$ be a bounded weak solution of~\eqref{eq:PbLambda}. Since the trace of $u$ belongs to $L^\infty(\partial\Omega)$, and both $u$ and $f(u)$ are in $L^\infty(\Omega)$, it follows from~\cite[Theorem~2.1]{Arc-Ros} that $u\in C^{0,\eta}(\overline{\Omega})$ for any $\eta\in(0,1)$. As $f$ is locally Lipschitz, then Schauder estimates (see~\cite[Theorem~6.1]{Gil-Trud}) immediately imply that $u\in C^{2,\eta}(\overline \Omega)$.
\end{remark}

In practice, we always deal with bounded solutions, where both concepts of solution can be used interchangeably. Specifically, the solutions we obtain are always nonnegative and less than $\beta$. Therefore, we can assume without loss of generality that $f$ is a globally Lipschitz function. This implies the existence of some $M>0$ such that 
    \begin{equation}
        \label{eq:f_incr}
        f(s)+Ms \text{ is increasing in } \re.
    \end{equation}

\subsection{Consequences of the strong maximum principle}

Here, we prove some results that can be derived from the maximum principle. Although we believe that the following result is well-known, we include its proof for the sake of completeness.

\begin{proposition}
\label{prop:comparison}
    Let $\gamma\geq 0$. Assume that $f$ is a Lipschitz function. Let $w\in C^2(\overline\Omega)$ be a subsolution of~\eqref{eq:PbGamma} and let $v\in C^2(\overline\Omega)$ be a supersolution. If $w\leq v$ in $\Omega$ and $w\not\equiv v$, then $w<v$ in $\overline\Omega$.
\end{proposition}

\begin{proof}
    Let $M>0$ be the positive constant defined in~\eqref{eq:f_incr}. Taking into account that $w\leq v$ in $\Omega$, we can say that
    \[
    f(v)+Mv - f(w) - Mw \geq 0 \text{ in } \Omega
    \]

    In this way, $v-w$ verifies the problem
    \begin{equation}
    \label{eq:Pf_Comp_1}
    \begin{cases}
    -\Delta (v-w) + M(v-w) \geq 0 & \mbox{in} \; \Omega,\\
    \frac{\partial (v-w)}{\partial \nu} + \gamma (v-w) \geq 0 & \mbox{on} \; \partial \Omega.
    \end{cases} 
    \end{equation}

    By the strong maximum principle, $v-w$ reaches its minimum in $\overline\Omega$ at some point $x_0\in \partial \Omega$. As $w\leq v$ in $\overline\Omega$ we have two possibilities: either $(v-w)(x_0)>0$ or $(v-w)(x_0)=0$. To finish the proof, we show that the second case cannot happen. Suppose that $(v-w)(x_0)=0$. As $w\not\equiv v$, the Hopf Lemma (\cite[Lemma~3.4]{Gil-Trud}) implies that $\frac{\partial(v-w)}{\partial \nu}(x_0)<0$. But this is a contradiction since we know from the boundary condition of~\eqref{eq:Pf_Comp_1} that $\frac{\partial (v-w)}{\partial \nu}(x_0)\geq -\gamma (v-w)(x_0) =0$.
\end{proof}

From Proposition~\ref{prop:comparison}, we can easily deduce the next result. 

\begin{corollary}
\label{cor:comparison}
    Let $\gamma>0$. Assume that $f$ is a Lipschitz function such that $f(\sigma)\leq 0$ for some $\sigma>0$. Then, there is no nonnegative solution $u\in C^2(\overline\Omega)$ to~\eqref{eq:PbGamma} such that $\|u\|_{C(\overline\Omega)} = \sigma$.
\end{corollary}

\begin{proof}
    Suppose that there is some $u\in C^2(\overline\Omega)$ nonnegative solution of~\eqref{eq:PbGamma} with $\|u\|_{C(\overline\Omega)} = \sigma$. As $v\equiv \sigma$ is a strict supersolution of~\eqref{eq:PbGamma}, and as $u\leq \sigma$ in $\Omega$, using Proposition~\ref{prop:comparison} we have that $u<\sigma$ in $\overline\Omega$, but this is a contradiction with $\|u\|_{C(\overline\Omega)} = \sigma$.
\end{proof}

\begin{remark}
    The same result holds for the Dirichlet problem~\eqref{eq:PbDirich} (see, for instance,~\cite[Lemma~6.2]{Amb-Hess}).
\end{remark}

In the following, we prove a pair of results related to the Leray-Schauder degree. For a nice introduction to this topic, we refer the reader to~\cite[Chapter~4]{Amb-Arc}.

Let $M>0$ be the one defined in~\eqref{eq:f_incr}. For any fixed $\lambda, \gamma>0$, we define the operator $K \colon \co \to C^2(\overline\Omega)$ as $K(w_1) = w_2$ for any $w_1\in \co$, where $w_2\in C^2(\overline\Omega)$ is the unique solution of
\begin{equation}
\label{eq:def_K}
\begin{cases}
\displaystyle -\Delta w_2 + \lambda Mw_2 = \lambda f(w_1) + \lambda Mw_1 & \mbox{in} \; \Omega,\\
\frac{\partial w_2}{\partial \nu} + \gamma w_2 = 0 & \mbox{on} \; \partial \Omega.
\end{cases} 
\end{equation}
This operator, which is well-defined, is compact in $\co$. Observe that a function $u\in C(\overline\Omega)$ is a solution of~\eqref{eq:PbLambda} if and only if $u$ is a fixed point of $K$, i.e., if $K(u)=u$.

\begin{lemma}
\label{lem:LS_1}
Suppose that $f$ is a Lipschitz function such that $f(\alpha)=f(\beta)=0$, where $0<\alpha<\beta$. If
\[
\mathcal{O}_{\alpha\beta} := \{u\in\co : u\geq 0 \text{ and } \alpha<\|u\|_\co<\beta\},
\]
then $\operatorname{deg}(I-K, \mathcal{O}_{\alpha\beta}, 0)$ is well-defined and its value is 0.
\end{lemma}

\begin{proof}
For any $\mu\in [0,1]$, we define the operator $K_\mu \colon \co \to C^2(\overline\Omega)$ as $K_\mu(w_1) = w_2$ for any $w_1\in \co$, where $w_2\in C^2(\overline\Omega)$ is the unique solution of
\begin{equation*}
\begin{cases}
-\Delta w_2 + \mu\lambda Mw_2 = \mu\lambda f(w_1) + \mu\lambda Mw_1 & \mbox{in} \; \Omega,\\
\frac{\partial w_2}{\partial \nu} + \gamma w_2 = 0 & \mbox{on} \; \partial \Omega.
\end{cases} 
\end{equation*}
These operators are well-defined and compact in $\co$. We point out that $K_0$ is just the zero operator and that $K_1=K$, where $K$ is the operator defined in~\eqref{eq:def_K}. Observe that a function $u\in C(\overline\Omega)$ is a solution of $(P_{\mu\lambda,\gamma})$ if and only if $u$ is a fixed point of $K_\mu$, i.e., if $K_\mu(u)=u$. In particular, any solution of~\eqref{eq:PbLambda} is a fixed point of $K$ and vice versa.

Due to Corollary~\ref{cor:comparison}, given any $\mu\in[0,1]$ we can assure that $K_\mu(w)\neq w$ for every $w\in \partial \mathcal{O}_{\alpha\beta}$. Therefore, the Leray-Schauder degree of these operators is well-defined in $\mathcal{O}_{\alpha\beta}$ and we can say that $\operatorname{deg}(I - K_\mu, \mathcal{O}_{\alpha\beta}, 0)$ is independent of $\mu\in[0,1]$. As a consequence, we obtain that
\begin{equation*}
\operatorname{deg}(I-K, \mathcal{O}_{\alpha\beta}, 0) = \operatorname{deg}(I, \mathcal{O}_{\alpha\beta}, 0) = 0. \qedhere
\end{equation*}
\end{proof}

\begin{lemma}
\label{lem:LS_2}
Suppose that $f$ is a Lipschitz function. Assume that $w\in C^2(\overline\Omega)$ is a strict subsolution of~\eqref{eq:PbGamma} and that $v\in C^2(\overline\Omega)$ is a strict supersolution. If
\[
\mathcal{O}_1 := \{u\in\co : w(x)<u(x)<v(x) \text{ in } \overline\Omega\},
\]
then $\operatorname{deg}(I-K, \mathcal{O}_1, 0)$ is well-defined and its value is 1.
\end{lemma}

\begin{proof}
First, we claim that $K$ maps $\overline{\mathcal{O}_1}$ into $\mathcal{O}_1$. Indeed, let us take $u_1\in \overline{\mathcal{O}_1}$ and denote $u_2:=K(u_1)$. We have that $w(x)\leq u_1(x)\leq v(x) \text{ in } \overline\Omega$ and we want to show that $w(x)< u_2(x) < v(x) \text{ in } \overline\Omega$. We are proving only the first inequality as the other one follows in a similar way. First, notice that
\[
f(u_1)+Mu_1 - f(w) - Mw \geq 0 \text{ in } \Omega
\]
as a consequence of~\eqref{eq:f_incr}. Moreover, $u_2\not\equiv w$. Suppose instead that $u_2\equiv w$. In this case, $f(u_1)+Mu_1 - f(w) - Mw \equiv 0$ and thus $u_1\equiv u_2\equiv w$. This implies that $w$ is solution~\eqref{eq:PbGamma}, but this is a contradiction since $w$ is a strict subsolution.

Therefore, $u_2-w$ verifies problem
\begin{equation*}
\begin{cases}
-\Delta (u_2-w) + M(u_2-w) \geq 0 & \mbox{in} \; \Omega,\\
\frac{\partial (u_2-w)}{\partial \nu} + \gamma (u_2-w) \geq 0 & \mbox{on} \; \partial \Omega.
\end{cases} 
\end{equation*}

From now on, we can argue as in the proof of Proposition~\ref{prop:comparison} to show that $w(x)<u_2(x)$ in $\overline\Omega$. As the same reasoning can be applied to deduce that $u_2(x)<v(x)$ in $\overline\Omega$, we have obtain that $u_2=K(u_1)\in \mathcal{O}_1$.

Finally, as the set $\mathcal{O}_1$ is convex, given any $\psi\in \mathcal{O}_1$ we can assure that $sK(w) + (1-s)\psi \in \mathcal{O}_1$ for any $w\in \overline{\mathcal{O}_1}$ and any $s\in [0,1]$. In particular, $sK(w) + (1-s)\psi \neq w$ for any $w\in\partial \mathcal{O}_1$ and any $s\in [0,1]$. Thus, for any $s\in [0,1]$ the Leray-Schauder degree $\operatorname{deg}(I-[sK+(1-s)\psi], \mathcal{O}_1, 0)$ is well-defined and independent of $s$. Then, we have
\begin{equation*}
\operatorname{deg}(I-K, \mathcal{O}_1, 0) = \operatorname{deg}(I-\psi, \mathcal{O}_1, 0) = 1. \qedhere
\end{equation*}

\end{proof}

\subsection{A stability result}

Several times, we need to pass to the limit in problems like $(P_{\lambda,\gamma_n})$, where $\gamma_n$ can converge or diverge. Since we are only interested in nonnegative solutions whose maximum remains below $\beta$, we can take advantage of the fact that these solutions are always bounded in $\co$. While proving that the limit satisfies either~\eqref{eq:PbLambda} (if $\gamma_n$ converges) or~\eqref{eq:PbDirich} (if $\gamma_n$ diverges) is straightforward, showing the strong convergence in $\co$ is a more delicate issue. 

This last point is key for our study. Given a sequence of solutions to $(P_{\lambda,\gamma_n})$ contained in $\mathcal{O}_{\alpha\beta}$, the convergence in $\co$ ensures that the limit belongs to $\overline{\mathcal{O}_{\alpha\beta}}$. Then, by Corollary~\ref{cor:comparison}, one can conclude that the limit is actually in $\mathcal{O}_{\alpha\beta}$.

When $\gamma_n$ converges, the convergence in $\co$ follows from the fact that these solutions can be shown to be uniformly bounded in $C^{0,\eta}(\overline\Omega)$ by applying~\cite[Theorem~2.1]{Arc-Ros}. Instead, when $\gamma_n$ diverges, we employ a strategy introduced in~\cite{Arc-Rez-Sil}, which is based on the Stampacchia regularity method (\cite{Stamp}). To this end, we state the following lemma, whose proof can be found in~\cite[Lemma~4.1]{Stamp}.

\begin{lemma}
    \label{lem:Stamp}
    Let $k_0\geq 0$ and let $\Psi : [k_0, \infty) \to [0, \infty)$ be a nonincreasing function. If there exist $c, a > 0$ and $b > 1$ such that
    \[
    \Psi(h) \leq \frac{c}{(h - k)^a} [\Psi(k)]^b, \quad \forall h > k > k_0,
    \]
    then
    \[
    \Psi(k_0 + d) = 0, \quad \text{for } d^a := c 2^{\frac{ab}{b - 1}} [\Psi(k_0)]^{b - 1}.
    \]
\end{lemma}

Now, we are ready to prove the stability result.

\begin{lemma}
    \label{lem:stability}
    Assume that $f\colon \R\to \R$ is a Lipschitz function. Let $u_n\in C^2(\overline\Omega)$ be a solution of
    \begin{equation}
	\label{eq:Pf_Lem_stab}
	\begin{cases}
		-\Delta u_n = \lambda_n f(u_n) & \text{in } \Omega,\\
		\frac{\partial u_n}{\partial \nu} + \gamma_n u_n=0 & \text{on } \partial\Omega,
	\end{cases}
\end{equation}
    where $\lambda_n$ and $\gamma_n$ are sequences of nonnegative numbers. Assume that $\lambda_n\to \lambda\in [0,\infty)$ and that $\gamma_n\to \gamma \in [0,\infty]$. If $u_n$ is bounded in $\co$, then, up to a subsequence, $u_n\to u$ in $\co$, where $u\in C^2(\overline\Omega)$ is
    \begin{enumerate}[i)]
        \item a solution of~\eqref{eq:PbLambda} if $\gamma<\infty$,
        \item a solution of the Dirichlet problem~\eqref{eq:PbDirich} if $\gamma=\infty$.
    \end{enumerate}
\end{lemma}

\begin{proof}
First of all, we can take $u_n$ as test function in the weak formulation of~\eqref{eq:Pf_Lem_stab} (see~\eqref{eq:weak_form}) to obtain that
\begin{equation}
    \label{eq:Pf_Lem_stab_1}
    \into |\nabla u_n|^2 + \gamma_n \intdo u_n^2  = \lambda_n\into f(u_n)u_n.
\end{equation}
As $f(u_n)u_n$ is bounded in $L^\infty(\Omega)$, we get that $\into |\nabla u_n|^2<C$. Since $u_n$ is also bounded in $L^\infty(\Omega)$, then we have that $u_n$ is bounded in $H^1(\Omega)$. Therefore, $u_n$ has a subsequence, still denoted by $u_n$, such that $u_n\rightharpoonup u$ weakly in $H^1(\Omega)$ to some $u\in H^1(\Omega)$. It also holds that $u_n\to u$ in $L^1(\Omega)$ and that $u_n\to u$ $\ae$ in $\Omega$.

From now on, we divide the proof into two parts.
\smallskip

\textbf{Case $\boldsymbol{\gamma_n\to \gamma<\infty}$:}

Here, we can easily pass to the limit in the weak formulation of~\eqref{eq:Pf_Lem_stab} to show that $u$ is a weak (and then a classical) solution of~\eqref{eq:PbLambda}.

To show the convergence of a subsequence of $u_n$ to $u$ in $\co$, we take advantage of the fact that $-\gamma_n u_n$ is bounded in $L^\infty(\partial\Omega)$. Then, applying~\cite[Theorem~2.1]{Arc-Ros}, we deduce that the sequence $u_n$ is bounded in $C^{0,\eta}(\overline\Omega)$ for any $\eta\in(0,1)$. Consequently, Arzel\`{a}-Ascoli Theorem ensures the existence of a subsequence of $u_n$ that converges to $u$ in $\co$.
\smallskip

\textbf{Case $\boldsymbol{\gamma_n\to \infty}$:}

Dropping a nonnegative term in~\eqref{eq:Pf_Lem_stab_1} and using that $f(u_n)u_n$ is bounded in $L^\infty(\Omega)$, we deduce that 
\[
\gamma_n \intdo u_n^2  < C.
\]
As $\gamma_n\to \infty$, we obtain that $\intdo u_n^2  \to 0$. Therefore, $u_n\to 0$ in $L^2(\partial \Omega)$ and, as the trace embedding $H^1(\Omega) \hookrightarrow L^2(\partial\Omega)$ is compact, then the trace of $u$ is $0$ and thus $u\in \huz$. From the weak formulation of~\eqref{eq:Pf_Lem_stab}, we have that
\[
\into \nabla u_n \nabla \varphi = \lambda_n \into f(u_n)\varphi,\ \forall \varphi\in \huz.
\]
Now we can easily pass to the limit to obtain that $u$ is a weak (and then a classical) solution of the Dirichlet problem~\eqref{eq:PbDirich}.

It remains for us to show that, up to a subsequence, $u_n$ converges in $\co$ to $u$. Since we make use of the Sobolev embeddings, we present the proof only for the case $N\geq 3$; the case $N=2$ can be treated similarly, while for $N=1$ the result follows directly from Morrey’s Theorem. We can assume that there is some $\gamma_0>0$ such that $\gamma_n\geq \gamma_0$ for every $n\in\N$. Recall that, in $\hu$, the norm defined by
\[
\|v\|_{\hu} := \left(\into |\nabla v|^2 + \gamma_0 \intdo v^2 \right)^\frac{1}{2}, \ v\in\hu,
\]
is equivalent to the usual norm and, consequently, the classical embeddings for $\hu$ remain valid with this norm. In particular, recall that the embeddings $H^1(\Omega) \hookrightarrow L^p(\Omega)$ and $H^1(\Omega) \hookrightarrow L^q(\partial\Omega)$ are continuous for any $p\in \left[1,\frac{2N}{N-2} \right]$ and $q\in \left[1, \frac{2(N-1)}{N-2}\right]$ (see~\cite[Chapter~2]{Necas}).

In what follows, we adapt an idea of~\cite{Arc-Rez-Sil}, which is based on the Stampacchia regularity method. For any $k>0$, we define the function $G_k(s):=\max\{\min\{s+k,0\},s-k\}$. From this point on, we fix $k>0$ and $n\in\N$. Taking $G_k(u_n-u)\in \hu$ as test function in the problem satisfied by $u_n$ (i.e., in~\eqref{eq:Pf_Lem_stab}), we obtain
\begin{equation}
    \label{eq:Pf_Lem_stab_2}
    \into \nabla u_n \nabla G_k(u_n-u) + \intdo \gamma_n u_n G_k(u_n-u)  = \into \lambda_nf(u_n) G_k(u_n-u).
\end{equation}
Now, we multiply the equation in~\eqref{eq:PbDirich} by $G_k(u_n-u)$ and, since $u\in H^2(\Omega)$ (in fact, $u\in C^2(\overline \Omega)$), we can apply the divergence theorem to obtain
\begin{equation}
    \label{eq:Pf_Lem_stab_3}
    \into \nabla u \nabla G_k(u_n-u) - \intdo \frac{\partial u}{\partial \nu} G_k(u_n-u)  = \into \lambda f(u) G_k(u_n-u).
\end{equation}
Subtracting both~\eqref{eq:Pf_Lem_stab_2} and~\eqref{eq:Pf_Lem_stab_3} we arrive at
\begin{equation}
    \label{eq:Pf_Lem_stab_4}
    \into |\nabla G_k(u_n-u)|^2 + \intdo \gamma_n u_n G_k(u_n-u)  = \into (\lambda_nf(u_n)-\lambda f(u)) G_k(u_n-u) - \intdo \frac{\partial u}{\partial \nu} G_k(u_n-u).
\end{equation}

With respect to the left-hand side of~\eqref{eq:Pf_Lem_stab_4}, as $u\equiv 0$ in $\partial\Omega$ and $sG_k(s)\geq G_k(s)^2$ for all $s\in\R$, we easily deduce that
\begin{equation}
\begin{split}
    \label{eq:Pf_Lem_stab_5}
    \into |\nabla G_k(u_n-u)|^2 + \intdo \gamma_n u_n G_k(u_n-u) &= \into |\nabla G_k(u_n-u)|^2 + \intdo \gamma_n (u_n-u) G_k(u_n-u)\\
    &\geq \into |\nabla G_k(u_n-u)|^2 + \intdo \gamma_0 G_k(u_n-u)^2\\
    &= \|G_k(u_n-u)\|_{H^1(\Omega)}^2.
\end{split}
\end{equation}

To simplify the computations related to the right-hand side of~\eqref{eq:Pf_Lem_stab_4}, we introduce some notation. Let
\[
A_n(k):= \{x\in\Omega: |u_n(x)-u(x)|>k\}, \qquad B_n(k):= \{x\in\partial\Omega: |u_n(x)-u(x)|>k\}.
\]
Since no ambiguity can arise, we use the same symbol $|\cdot|$ to denote both the $N$-dimensional Lebesgue measure (for subsets of $\Omega$) and the $(N-1)$-dimensional Hausdorff measure (for subsets of $\partial\Omega$). We also set $q:= \frac{2(N-1)}{N-2}$.

Regarding the third integral in~\eqref{eq:Pf_Lem_stab_4}, since the term $\lambda_nf(u_n) - \lambda f(u)$ is bounded in $L^\infty(\Omega)$, say by $C_1>0$, we can apply H\"{o}lder's inequality, the continuous embedding $H^1(\Omega) \hookrightarrow L^q(\Omega)$ and Young's inequality to deduce
\begin{equation}
\begin{split}
    \label{eq:Pf_Lem_stab_6}
    \into (\lambda_nf(u_n)-\lambda f(u)) G_k(u_n-u) &\leq C_1 \into |G_k(u_n-u)| = C_1 \int_{A_n(k)} |G_k(u_n-u)|\\
    &\leq C_1 |A_n(k)|^{\frac{1}{q'}} \|G_k(u_n-u)\|_{L^q(\Omega)} \leq C_2 |A_n(k)|^{\frac{1}{q'}} \|G_k(u_n-u)\|_{H^1(\Omega)}\\
    &\leq \frac{1}{4}\|G_k(u_n-u)\|_{H^1(\Omega)}^2 + C_3|A_n(k)|^{\frac{2}{q'}},
\end{split}
\end{equation}
for some constants $C_2,C_3>0$. A similar argument applies to the last integral in~\eqref{eq:Pf_Lem_stab_4}, since $\frac{\partial u}{\partial \nu}\in L^\infty(\partial\Omega)$, although in this case we rely on the continuous embedding $H^1(\Omega) \hookrightarrow L^q(\partial\Omega)$. Reasoning as before, we obtain for some constant $C_4>0$ that
\begin{equation}
    \label{eq:Pf_Lem_stab_7}
    - \intdo \frac{\partial u}{\partial \nu} G_k(u_n-u) \leq \frac{1}{4}\|G_k(u_n-u)\|_{H^1(\Omega)}^2 + C_4|B_n(k)|^{\frac{2}{q'}}.
\end{equation}

Returning to the equality~\eqref{eq:Pf_Lem_stab_4}, we can now combine~\eqref{eq:Pf_Lem_stab_5},~\eqref{eq:Pf_Lem_stab_6} and~\eqref{eq:Pf_Lem_stab_7} to obtain that
\begin{equation}
    \label{eq:Pf_Lem_stab_8}
    \|G_k(u_n-u)\|_{H^1(\Omega)}^2 \leq 2C_3|A_n(k)|^{\frac{2}{q'}} + 2C_4|B_n(k)|^{\frac{2}{q'}} \leq C_5 \left(|A_n(k)|^{\frac{2}{q'}} + |B_n(k)|^{\frac{2}{q'}} \right),
\end{equation}
where $C_5>0$. Using the Sobolev embeddings $H^1(\Omega) \hookrightarrow L^q(\Omega)$ and $H^1(\Omega) \hookrightarrow L^q(\partial\Omega)$, we can find some $C_6>0$ such that
\begin{equation}
    \label{eq:Pf_Lem_stab_9}
    \|G_k(u_n-u)\|_{H^1(\Omega)}^2 \geq C_6 \left( \|G_k(u_n-u)\|_{L^q(\Omega)}^2 + \|G_k(u_n-u)\|_{L^q(\partial\Omega)}^2 \right).
\end{equation}
Observe that, if we take $h>k>0$, then $A_n(h)\subseteq A_n(k)$ and $B_n(h)\subseteq B_n(k)$. Taking into account that $|G_k(s)|\geq h-k$ for any $s\in\R$ such that $|s|\geq h$, we deduce that
\begin{equation}
\begin{split}
    \label{eq:Pf_Lem_stab_10}
    & \|G_k(u_n-u)\|_{L^q(\Omega)}^2 = \left(\int_{A_n(k)} |G_k(u_n-u)|^q\right)^\frac{2}{q} \geq (h-k)^2 |A_n(h)|^\frac{2}{q},\\
    & \|G_k(u_n-u)\|_{L^q(\partial\Omega)}^2 = \left(\int_{B_n(k)} |G_k(u_n-u)|^q\right)^\frac{2}{q} \geq (h-k)^2 |B_n(h)|^\frac{2}{q}.
\end{split}
\end{equation}
Joining~\eqref{eq:Pf_Lem_stab_8},~\eqref{eq:Pf_Lem_stab_9} and~\eqref{eq:Pf_Lem_stab_10}, we obtain
\begin{equation}
    \label{eq:Pf_Lem_stab_11}
    C_6 (h-k)^2 \left( |A_n(h)|^\frac{2}{q} + |B_n(h)|^\frac{2}{q} \right) \leq C_5 \left(|A_n(k)|^{\frac{2}{q'}} + |B_n(k)|^{\frac{2}{q'}} \right).
\end{equation}
Recall that for any $s,t\geq0$ and $p>0$, the following inequality holds: $\min\{1,2^{p-1}\}(s^p+t^p) \leq (s+t)^p \leq \max\{1,2^{p-1}\}(s^p+t^p)$. Applying this algebraic inequality to~\eqref{eq:Pf_Lem_stab_11}, we deduce for some $C_7>0$ that
\begin{equation}
    \label{eq:Pf_Lem_stab_12}
    |A_n(h)| + |B_n(h)| \leq \frac{C_7}{(h-k)^q} \left(|A_n(k)| + |B_n(k)| \right)^{\frac{q}{q'}}.
\end{equation}

Let $\varepsilon>0$ fixed. We consider the function $\Psi_{n}\colon [0,\infty)\to [0,\infty)$ defined as $\Psi_{n}(k):=|A_n(k)| + |B_n(k)|$ for any $k\geq 0$. As $\Psi_{n}$ is nonincreasing, due to relation~\eqref{eq:Pf_Lem_stab_12} we can apply Lemma~\ref{lem:Stamp} with $k_0=\varepsilon$, $a=q =\frac{2(N-1)}{N-2}>0$, $c=C_7>0$ and $b=\frac{q}{q'}=\frac{N}{N-2}>1$ to obtain, for every $n\in\na$, that
\begin{equation}
    \label{eq:Pf_Lem_stab_13}
    \Psi_{n}(\varepsilon + d_{n}) = |A_n(\varepsilon + d_{n})| + |B_n(\varepsilon + d_{n})| = 0, \text{ with } d_{n}^a = c2^\frac{ab}{b-1} \left(|A_{n}(\varepsilon)| + |B_{n}(\varepsilon)| \right)^{b-1}.
\end{equation}
Notice that, by the definition of $A_{n}(k)$ and $B_n(k)$, for every $n\in\N$ we have
\[
\varepsilon|A_{n}(\varepsilon)| \leq \int_{A_{n}(\varepsilon)} |u_n-u| \leq \|u_n-u\|_{L^1(\Omega)} \qquad \text{and} \qquad \varepsilon|B_{n}(\varepsilon)| \leq \int_{B_{n}(\varepsilon)} |u_n-u| \leq \|u_n-u\|_{L^1(\partial\Omega)}.
\]
Since $u_n$ converges strongly to $u$ in $\luo$ and in $L^1(\partial\Omega)$, we can find $n_0\in\na$ (depending on $\varepsilon$) such that $d_{n}<\varepsilon$ for all $n\geq n_0$. As $\Psi_{n}$ is nonincreasing, from~\eqref{eq:Pf_Lem_stab_13} we deduce that $|A_n(2\varepsilon)| + |B_n(2\varepsilon)| = 0$ for every $n\geq n_0$ or, equivalently, that
\[
\|u_n-u\|_\co \leq 2\varepsilon, \ \forall n\geq n_0.
\]
Then, we can conclude that $u_n$ converges strongly in $\co$ to $u$, where $u$ is the solution that we have found for~\eqref{eq:PbDirich}.
\end{proof}

\begin{remark}
    In Lemma~\ref{lem:stability} above, we have assumed that $f\colon \R\to\R$ is Lipschitz continuous for simplicity, in order to ensure the equivalence between classical and bounded weak solutions. However, the same statement remains valid in the setting of bounded weak solutions under the weaker assumption that $f$ is merely continuous.
\end{remark}

\begin{remark}
    Up to our knowledge, the result concerning the case $\gamma_n\to \infty$ is completely new in the literature. Only in the particular case $N=3$, this result can also be deduced from~\cite{Amrouche}, where the authors establish uniform $W^{1,p}$-estimates for $u_n$ using quite different techniques.
\end{remark}

\section{An auxiliary problem}
\label{sec:aux}

Throughout this section, we assume that $f\colon \R\to\R$ is a locally Lipschitz function with two consecutive zeros, $0<\alpha<\beta$, such that $f(s)>0$ when $s\in (\alpha,\beta)$. Here, we truncate $f$ to define the continuous function
\begin{equation*}
\tilde f(s)=\begin{cases}
f(s) & \text{ if }  s\in (\alpha,\beta), \\
0 & \text{ if } s\not\in (\alpha,\beta),\\
\end{cases} 
\end{equation*}
and we study some properties of the solutions to problem
\begin{equation}
\label{eq:PbAux}
\begin{cases}
-\Delta u = \lambda\tilde f(u) & \mbox{in} \; \Omega,\\
\frac{\partial u}{\partial \nu} + \gamma u = 0 & \mbox{on} \; \partial \Omega.
\end{cases} 
\end{equation}

For any fixed $\gamma>0$, our aim is to prove that the maximal solution $\ou_\lambda$ to~\eqref{eq:PbAux} in the interval $[0,\beta]$ converges uniformly in $\overline\Omega$ to $\beta$ as $\lambda$ diverges. Therefore, for large values of $\lambda$, $\ou_\lambda$ will be also a solution of~\eqref{eq:PbLambda}.

We start proving the existence of $\ou_\lambda$ for large $\lambda$ using variational methods. We also prove that $\ou_\lambda$ is an increasing sequence in $\lambda$, in the sense that $\ou_\lambda(x)$ is an increasing sequence of real numbers for every $x\in \overline\Omega$.

\begin{lemma}
    \label{lem:exist_gam_fix}
    Let $\gamma>0$ be fixed. Then, there exists some $\overline \lambda(\gamma)>0$ such that a positive maximal solution $\ou_\lambda$ to $\eqref{eq:PbAux}$ in the interval $[0,\beta]$ exists for any $\lambda>\overline\lambda$. Furthermore, $\ou_\lambda(x)$ is increasing in $\lambda$ for any $x\in\overline\Omega$.
\end{lemma}

\begin{proof}
We point out that every nontrivial solution of~\eqref{eq:PbAux} is positive and has its maximum in the interval $(\alpha, \beta)$ due to Proposition~\ref{prop:comparison} and Corollary~\ref{cor:comparison}. In addition, solutions of~\eqref{eq:PbAux} correspond to critical points of the energy functional
\begin{equation*}
    \tilde{I}_{\lambda,\gamma}(u) = \frac{1}{2} \into |\nabla u|^2 + \frac{\gamma}{2} \intdo u^2  - \lambda\into \tilde{F}(u),\ \forall u\in\hu,
\end{equation*}
where $\tilde{F}(s) = \int_0^s \tilde{f}(t)\ \mathrm{d}t$.  Since $\tilde F$ is continuous and bounded, the functional $\tilde I_\gamma$ is coercive and weakly lower semicontinuous. Therefore, we can deduce that $\tilde I_\gamma$ has a global minimum. Now, we prove that this minimum is not reached at the trivial function for large $\lambda$. Indeed, taking into account that $\tilde F(\beta)>0$, one has
\[
\tilde I_{\lambda,\gamma} (\beta) = \gamma \frac{\beta^2 |\partial \Omega|}{2} - \lambda|\Omega| \tilde F(\beta) := C_1\gamma - C_2 \lambda,
\]
where $C_1,C_2>0$, so one can take 
\begin{equation*}
\overline\lambda>\frac{C_1\gamma}{C_2}
\end{equation*}
and then for every $\lambda > \overline\lambda$ one has $\tilde I_{\lambda,\gamma} (\beta)<0=\tilde I_{\lambda,\gamma} (0)$.

Then, we have proved the existence of a positive solution to~\eqref{eq:PbAux} with maximum in $(\alpha,\beta)$. As $\beta$ is always a supersolution of~\eqref{eq:PbAux}, the existence of a maximal solution $\ou_\lambda$ of~\eqref{eq:PbAux} in the interval $[0,\beta]$ can be deduced using a standard iterative scheme starting from $\beta$.

The fact that $\ou_\lambda$ is increasing in $\lambda$ can be easily deduced since $\tilde f\geq 0$. Indeed, given $\lambda_1<\lambda_2$, one has that $\ou_{\lambda_1}$ is a strict subsolution of~\eqref{eq:PbAux} with $\lambda=\lambda_2$, so the maximality of $\ou_{\lambda_2}$ and Proposition~\ref{prop:comparison} imply that $\ou_{\lambda_1}(x) < \ou_{\lambda_2}(x)$ for any $x\in\overline\Omega$.
\end{proof}

As mentioned earlier, our aim is to show that $\ou_\lambda\to \beta$ uniformly in $\overline\Omega$ when $\lambda\to\infty$. We begin by proving this result when $\Omega$ is a ball $B_R(0)$ of radius $R>0$.

\begin{lemma}
\label{lem:ball}
For any $\varepsilon\in (0,\beta)$, there exists some $\overline\lambda (\varepsilon)$ such that the Dirichlet problem
\begin{equation}
\label{eq:PbAuxDirBall}
\begin{cases}
-\Delta w = \lambda\tilde f(w) & \mbox{in} \; B_R(0),\\
w = \beta-\varepsilon & \mbox{on} \; \partial B_R(0),\\
w(x)\in (\beta-\varepsilon, \beta) & \mbox{for any} \; x\in B_R(0),
\end{cases} 
\end{equation}
has a radial solution $w_\lambda\in C^2(\overline \Omega)$ for any $\lambda>\overline\lambda$. Moreover, $w_\lambda'(R)\to -\infty$ as $\lambda\to\infty$.
\end{lemma}

\begin{remark}
This result immediately implies that $\ou_\lambda\to \beta$ uniformly in $\overline B_R(0)$ when $\lambda\to\infty$, where $\ou_\lambda$ denotes the solution to $\eqref{eq:PbAux}$ with $\Omega=\overline B_R(0)$ given by Lemma~\ref{lem:exist_gam_fix}. Indeed, given any $\varepsilon\in (0,\beta)$, we can take $\lambda$ large so that $w_\lambda'(R) < -\gamma (\beta-\varepsilon)$. This ensures that $w_\lambda'(R)+\gamma w_\lambda(R)<0$, and thus $w_\lambda$ is a strict subsolution of~\eqref{eq:PbAux}. By
Proposition~\ref{prop:comparison}, it then follows that $w_\lambda<\ou_\lambda$ in $\overline B_R(0)$ for sufficiently large $\lambda$.
\end{remark}

\begin{proof}
Let $\varepsilon\in (0,\beta)$. We define $g(s):=\tilde f (s+\beta-\varepsilon)$ for any $s\in\R$. Observe that if $v_\lambda$ is a solution of the Dirichlet problem
\begin{equation}
\label{eq:Pf_Ball_1}
\begin{cases}
-\Delta v = \lambda g(v) & \mbox{in} \; B_R(0),\\
v = 0 & \mbox{on} \; \partial B_R(0),\\
v(x)\in (0, \varepsilon) & \mbox{for any} \; x\in B_R(0),
\end{cases} 
\end{equation}
then $w_\lambda := v_\lambda +\beta-\varepsilon$ is a solution of~\eqref{eq:PbAuxDirBall} with the same derivative on the boundary. Therefore, for our purposes, it suffices to study problem~\eqref{eq:Pf_Ball_1}.

The existence of a solution of~\eqref{eq:Pf_Ball_1} for large $\lambda$ is standard. Indeed, as $g(s)\geq 0$ for any $s\in \R$ and $g(\varepsilon)=0$, one can use, for instance,~\cite[Theorem~3]{B-GM-I}, to obtain the existence of some $\overline\lambda>0$ for which a maximal solution $v_\lambda$ to~\eqref{eq:Pf_Ball_1} exists for any $\lambda>\overline\lambda$. Moreover,~\cite[Theorem~3]{B-GM-I} also guarantees that $v_\lambda$ is an increasing sequence in $B_R(\Omega)$ that converges uniformly in compact subsets of $B_R(\Omega)$ to $\varepsilon$. Using the results of~\cite{G-Ni-N}, we can also ensure that $v_\lambda$ is radially symmetric and radially decreasing, so we can write $v_\lambda(r)$, where $r=|x|$.

It remains for us to show that $v'_\lambda(R)\to -\infty$ as $\lambda\to \infty$. The Pohozaev identity for radial solutions gives us the relation
\begin{equation}
\label{eq:Pf_Ball_3}
\left((v_\lambda'(R)\right)^2 = 2 R^{-N} \lambda \int_0^R r^{N-1} \left( NG(v_\lambda(r)) - \frac{N-2}{2} v_\lambda(r) g(v_\lambda(r)) \right) \ \mathrm{d}r,
\end{equation}
where $G(s) = \int_0^s g(t)\ \mathrm{d}t$. As $v_\lambda'(R)<0$ thanks to Hopf's Lemma, to our aim it suffices to show the existence of some $C>0$ such that
\begin{equation}
    \label{eq:Pf_Ball_2}
    \int_0^R r^{N-1} \left( NG(v_\lambda(r)) - \frac{N-2}{2} v_\lambda(r) g(v_\lambda(r)) \right) \ \mathrm{d}r > C
\end{equation}
for every $\lambda$ greater than some value. 

To simplify computations, we define the function $h(s) := NG(s) - \frac{N-2}{2} s g(s)$ for all $s\in [0,\varepsilon]$. Observe that $h(\varepsilon)= NG(\varepsilon)>0$ and, as $h$ is continuous, we can find some $\delta\in (0,\varepsilon)$ such that $h(s)>C_1>0$ for any $s\in [\varepsilon-\delta,\varepsilon]$. Furthermore, $h$, that can be negative, is bounded, so there is some $C_2>0$ such that $h(s)>-C_2$ for all $s\in [0,\varepsilon]$.

On the other hand, since $v_\lambda$ converges uniformly to $\varepsilon$  in compact subsets of $B_R(0)$, for each $n\in\N$ we can find $\lambda_n>0$ large such that $v_\lambda(r)\in [\varepsilon-\delta,\varepsilon]$ for every $0\leq r\leq R-\frac{R}{n}$ and for every $\lambda>\lambda_n$. Then, for any $\lambda>\lambda_n$ we obtain that
\begin{equation}
\label{eq:Pf_Ball_4}
\begin{split}
\int_0^R r^{N-1} h(v_\lambda(r)) \ \mathrm{d}r
& = \int_0^{R-\frac{R}{n}} r^{N-1} h(v_\lambda(r)) \ \mathrm{d}r + \int_{R-\frac{R}{n}}^R r^{N-1} h(v_\lambda(r)) \ \mathrm{d}r \\
&\geq C_1 \int_0^{R-\frac{R}{n}} r^{N-1} \ \mathrm{d}r - C_2 \int_{R-\frac{R}{n}}^R r^{N-1} \ \mathrm{d}r\\
&= \frac{1}{N} C_1 \left(R-\frac{R}{n}\right)^N - \frac{1}{N} C_2 \left(R^N - \left(R-\frac{R}{n}\right)^N\right).
\end{split}
\end{equation}
Since this last quantity converges to $\frac{1}{N}C_1R^N>0$ when $n\to\infty$, one can take $n_0$ large to obtain from~\eqref{eq:Pf_Ball_4} that~\eqref{eq:Pf_Ball_2} holds for any $\lambda>\lambda_{n_0}$. Therefore, from~\eqref{eq:Pf_Ball_3} we deduce that $v_\lambda'(R) \leq -\sqrt{2R^{-N}C} \sqrt{\lambda}$ for any $\lambda>\lambda_{n_0}$ and, consequently, we can conclude that $v_\lambda'(R)\to -\infty$ as $\lambda\to\infty$.
\end{proof}

Now, we return to consider a general bounded domain $\Omega$ with $C^2$ boundary. The next result provides the convergence in $C(\overline\Omega)$ of $\ou_\lambda$ to $\beta$. The proof, which uses an original approach, is based on geometrical arguments.

\begin{proposition}
    \label{lem:conv_gam_fix}
    Let $\gamma>0$ be fixed. Then, the sequence of solutions $\ou_\lambda$ to~\eqref{eq:PbAux} given by Lemma~\ref{lem:exist_gam_fix} converges uniformly in $\overline\Omega$ to $\beta$.
\end{proposition}

\begin{proof}
When $N=1$, this result follows directly from Lemma~\ref{lem:ball}. Therefore, in the following, we assume $N\geq 2$. We also point out that $\ou_\lambda$ is an increasing sequence of continuous functions in a compact domain, so due to Dini's Theorem it suffices to prove that $\ou_\lambda\to \beta$ pointwise in $\overline\Omega$.

In $\Omega$, the pointwise convergence of $\ou_\lambda$ to $\beta$ is a consequence of the well-known behaviour of the solutions to the Dirichlet problem
\begin{equation}
\label{eq:PbAuxDir}
\begin{cases}
-\Delta w = \lambda\tilde f(w) & \mbox{in} \; \Omega,\\
w = 0 & \mbox{on} \; \partial \Omega.
\end{cases} 
\end{equation}

To be concrete, in~\cite[Theorem~3]{B-GM-I}, it is proved that a positive solution $w_\lambda$ to~\eqref{eq:PbAuxDir} with maximum in $(\alpha,\beta)$ exists for large $\lambda$ and that $w_\lambda\to \beta$ uniformly in compact subsets of $\Omega$. As $w_\lambda$ is a strict subsolution of $\eqref{eq:PbAux}$, then $w_\lambda < \ou_\lambda$ thanks to Proposition~\ref{prop:comparison} and thus $\overline u_\lambda$ also converges to $\beta$ uniformly in compact subsets of $\Omega$. In particular, $\overline u_\lambda(x)$ converges to $\beta$ for any $x\in \Omega$.

It remains for us to show that $\overline u_\lambda\to \beta$ pointwise in $\partial\Omega$. This will be done by constructing, for each small $\varepsilon>0$ and each point $y_0$ of the boundary, a subsolution of~\eqref{eq:PbAux} for sufficiently large $\lambda$ that takes the value $\beta-\varepsilon$ in $y_0$. Here, we will make use of Lemma~\ref{lem:ball}.

With this aim, we fix $y_0\in\partial\Omega$ and $\varepsilon>0$ small. As $\partial\Omega$ satisfies the interior ball condition, there is some $x_0\in \Omega$ and some $R_1>0$ such that $B_{R_1}(x_0)\subset \Omega$ and $\partial B_{R_1}(x_0) \cap \partial\Omega = \{y_0\}$. For any $x\in \R^N$, we define $r(x) = |x-x_0|$. Observe that, for any $y\in\partial\Omega$, we have that
\[
\frac{\partial r}{\partial \nu}(y) = \nabla r(y) \cdot \nu(y) = \frac{y-x_0}{|y-x_0|} \cdot \nu(y),
\]
where, we recall, $\nu(y)$ denotes the outward normal unit vector to $\partial\Omega$ at $y$. As $B_{R_1}(x_0)$ is tangent to $\partial\Omega$ at $y_0$, the outward normal vectors to both surfaces at $y_0$ must be equal. Then, we obtain
\[
\frac{\partial r}{\partial \nu}(y_0) = 1.
\]
As $\nu(y)$ is a continuous function (due to the smoothness of $\partial\Omega$), we can find $R_2>R_1$ such that
\begin{equation}
\label{eq:Pf_conv_gfix_1}
    \frac{\partial r}{\partial \nu}(y) \geq \frac{1}{2},\quad \forall y\in \Gamma_{1} := B_{R_2}(x_0) \cap \partial\Omega.
\end{equation}

Now, we construct the subsolution in the following way. Using Lemma~\ref{lem:ball} with $B_R(0)$ replaced by $B_{R_1}(x_0)$, we can find some $\overline\lambda_1(\varepsilon,y_0)>0$ such that a radial solution $w_\lambda(r)$ to~\eqref{eq:PbAuxDirBall} defined in $B_{R_1}(x_0)$ exists for every $\lambda>\overline\lambda_1$ (recall $r(x) = |x-x_0|$). Furthermore, it holds that $\lim_{\lambda\to\infty} w_\lambda'(R_1)=-\infty$. When $N=2$, we extend this function to the whole space $\R^N$ as
\[
\tilde w_\lambda (r) = \begin{cases}
    w_\lambda(r) & \text{ if } 0\leq r < R_1,\\
    \beta-\varepsilon + \big(\log r - \log R_1\big) R_1 w_\lambda'(R_1) & \text{ if } R_1\leq r < R_2,\\
    \beta-\varepsilon + \big(\log R_2 - \log R_1\big) R_1 w_\lambda'(R_1) & \text{ if } r\geq R_2,
\end{cases}
\]
and, when $N\geq 3$, as
\[
\tilde w_\lambda (r) = \begin{cases}
    w_\lambda(r) & \text{ if } 0\leq r < R_1,\\
    \beta-\varepsilon + \left(r^{2-N}-R_1^{2-N}\right) R_1^{N-1} (2-N)^{-1} w_\lambda'(R_1) & \text{ if } R_1\leq r < R_2,\\
    \beta-\varepsilon + \left(R_2^{2-N}-R_1^{2-N}\right) R_1^{N-1} (2-N)^{-1} w_\lambda'(R_1) & \text{ if } r\geq R_2.
\end{cases}
\]
Since $\lim_{\lambda\to\infty} w_\lambda'(R_1)=-\infty$, we can take $\overline \lambda_2>\overline\lambda_1$ such that $\tilde w_\lambda(R_2)<0$ for any $\lambda>\overline \lambda_2$.

We observe that $\tilde w_\lambda(r)$ is continuous, radially decreasing for $0\leq r\leq R_2$, and identically equal to a negative constant for $r>R_2$, provided that $\lambda>\overline \lambda_2$. Moreover, $\tilde w_\lambda(r)$ is harmonic when $R_1<r<R_2$. We claim that $\tilde w_\lambda$ is a weak subsolution of~\eqref{eq:PbAux} for sufficiently large $\lambda$. Then, we have to show that $\tilde w_\lambda\in H^1(\Omega)$ verifies
\begin{equation*}
    \into \nabla \tilde w_\lambda \nabla \varphi + \gamma \intdo \tilde w_\lambda \varphi \leq \lambda\into \tilde f(\tilde w_\lambda)\varphi,\ \forall \varphi\in \hu \text{ with } \varphi\geq 0
\end{equation*}
when $\lambda$ is large. From now on, we fix $\varphi\in \hu \text{ with } \varphi\geq 0$. To ease the notation, we set $\Gamma_2 = (\Omega\setminus B_{R_2}(x_0))\cap \partial \Omega$ and $\Gamma_3 = \partial B_{R_2}(x_0) \cap \Omega$ (see Figure~\ref{fig:dominio}).

\begin{figure}[ht]
\centering
  \includegraphics[width=6cm]{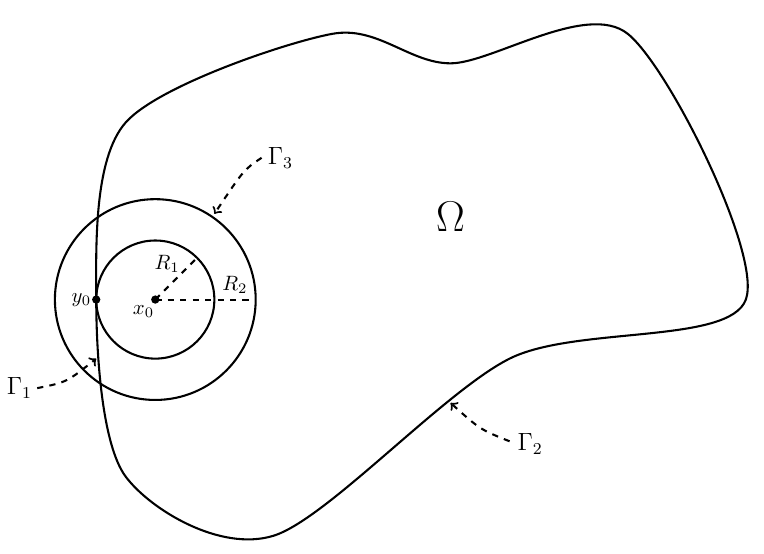}
\caption{\label{fig:dominio} Sketch of the different sets that appear in the proof of Proposition~\ref{lem:conv_gam_fix}.}
\end{figure}

Of course, $\tilde w_\lambda\in H^1(\Omega)$.  On the one hand, as $\tilde w_\lambda \in H^2 \left(B_{R_2}(x_0)\cap \Omega\right)$ because $\lim_{r\to R_1^-} \tilde w_\lambda'(r)$ is equal to $\lim_{r\to R_1^+} \tilde w_\lambda'(r)$, and as $-\Delta \tilde w_\lambda \leq \lambda \tilde f(\tilde w_\lambda)$ $\textrm{a.e.}$ in $B_{R_2}(x_0)\cap \Omega$, we can integrate by parts to obtain 
\begin{equation}
    \label{eq:Pf_conv_gfix_3}
    \int_{B_{R_2}(x_0)\cap \Omega} \nabla \tilde w_\lambda \nabla \varphi - \int_{\Gamma_1} \frac{\partial \tilde w_\lambda}{\partial \nu} \varphi - \int_{\Gamma_3} \tilde w_\lambda'(R_2^-) \varphi \leq \lambda\int_{B_{R_2}(x_0)\cap \Omega} \tilde f(\tilde w_\lambda)\varphi,
\end{equation}
where $\tilde w_\lambda'(R_2^-):=\lim_{r\to R_2^-} \tilde w_\lambda'(r)$. On the other hand, we obviously have $\tilde w_\lambda \in H^2 \left(\Omega\setminus B_{R_2}(x_0)\right)$ and $-\Delta \tilde w_\lambda = 0 \leq \lambda \tilde f(\tilde w_\lambda)$ in $\Omega\setminus B_{R_2}(x_0)$, so integration by parts yields
\begin{equation}
    \label{eq:Pf_conv_gfix_4}
    \int_{\Omega\setminus B_{R_2}(x_0)} \nabla \tilde w_\lambda \nabla \varphi - \int_{\Gamma_2} \frac{\partial \tilde w_\lambda}{\partial \nu} \varphi + \int_{\Gamma_3} \tilde w_\lambda'(R_2^+) \varphi \leq \lambda\int_{\Omega\setminus B_{R_2}(x_0)} \tilde f(\tilde w_\lambda)\varphi,
\end{equation}
where, as before, $\tilde w_\lambda'(R_2^+):=\lim_{r\to R_2^+} \tilde w_\lambda'(r)$. Taking into account that $\tilde w_\lambda'(R_2^-) = (R_1/R_2)^{N-1} w_\lambda'(R_1)<0 =\tilde w_\lambda'(R_2^+)$, we can sum both~\eqref{eq:Pf_conv_gfix_3} and~\eqref{eq:Pf_conv_gfix_4} to obtain, after dropping the nonnegative term related to $\Gamma_3$, that
\begin{equation}
    \label{eq:Pf_conv_gfix_5}
    \into \nabla \tilde w_\lambda \nabla \varphi - \int_{\Gamma_1} \frac{\partial \tilde w_\lambda}{\partial \nu} \varphi - \int_{\Gamma_2} \frac{\partial \tilde w_\lambda}{\partial \nu} \varphi \leq \lambda\into \tilde f(\tilde w_\lambda)\varphi.
\end{equation}

Now, we focus our attention on the boundary terms. When $y\in\Gamma_1$, we have that $r(y)\in[R_1,R_2)$. Here, we stress that $\tilde w_\lambda(y)\leq \beta-\varepsilon$ and that $w_\lambda'(r(y)) \leq (R_1/R_2)^{N-1} w_\lambda'(R_1)<0$. Using~\eqref{eq:Pf_conv_gfix_1}, we obtain
\begin{equation*}
    \begin{split}
    \frac{\partial \tilde w_\lambda}{\partial\nu}(y) + \gamma \tilde w_\lambda(y) &\leq \frac{\partial \tilde w_\lambda}{\partial\nu}(y) + \gamma (\beta-\varepsilon) = \tilde w_\lambda'(r(y)) \frac{\partial r}{\partial \nu}(y) + \gamma (\beta-\varepsilon)\\
    &\leq \frac{1}{2} \tilde w_\lambda'(r(y)) + \gamma (\beta-\varepsilon) \leq \frac{1}{2} (R_1/R_2)^{N-1} w_\lambda'(R_1) + \gamma (\beta-\varepsilon).
    \end{split}
\end{equation*}
Since $\lim_{\lambda\to\infty} w_\lambda'(R_1) = -\infty$, we can find $\overline\lambda_3>\overline\lambda_2$ such that, for any $\lambda>\overline\lambda_3$, it holds
\begin{equation}
    \label{eq:Pf_conv_gfix_6}
    \frac{\partial \tilde w_\lambda}{\partial\nu}(y) + \gamma \tilde w_\lambda(y) < 0 \text{ for all } y\in \Gamma_1.
\end{equation}
Besides, when $y\in \Gamma_2$ we have that $r(y)> R_2$, so $\tilde w_\lambda$ is negative on this set and its normal derivative is zero. Therefore, we obtain
\begin{equation}
    \label{eq:Pf_conv_gfix_7}
    \frac{\partial \tilde w_\lambda}{\partial\nu}(y) + \gamma \tilde w_\lambda(y) = \gamma \tilde w_\lambda(y) <0 \text{ for all } y\in \Gamma_2.
\end{equation}

Joining~\eqref{eq:Pf_conv_gfix_5},~\eqref{eq:Pf_conv_gfix_6} and~\eqref{eq:Pf_conv_gfix_7}, we deduce that $\tilde w_\lambda$ is a weak subsolution of~\eqref{eq:PbAux} provided $\lambda>\overline\lambda_2$. As $\overline u_\lambda$ is the maximal solution of~\eqref{eq:PbAux} in the interval $[0,\beta]$, we have that $\tilde w_\lambda(x) \leq \overline u_\lambda(x)$ for all $x\in \overline{\Omega}$ and all $\lambda>\overline\lambda_3(\varepsilon,y_0)$. In particular, $\tilde w_\lambda(y_0) = \beta-\varepsilon \leq \overline u_\lambda(y_0)$. Since $\varepsilon$ is arbitrary, we can affirm that $\overline u_\lambda(y_0)\to \beta$ as $\lambda\to\infty$ and this concludes the proof.
\end{proof}

In the following, we aim to study problem~\eqref{eq:PbAux} for a fixed $\lambda>0$ to determine whether a solution to~\eqref{eq:PbAux} exists when $\gamma$ is small and, if so, to study the limit of such solutions as $\gamma\to 0$. Since $\lambda$ does not play a significant role here, we consider instead problem
\begin{equation}
\label{eq:PbAux2}
\begin{cases}
-\Delta u = \tilde f(u) & \mbox{in} \; \Omega,\\
\frac{\partial u}{\partial \nu} + \gamma u = 0 & \mbox{on} \; \partial \Omega.
\end{cases} 
\end{equation}

\begin{proposition}
\label{prop:exist}
There is some constant $\overline\gamma>0$ such that, for any $0<\gamma<\overline\gamma$, problem~\eqref{eq:PbAux2} has a positive maximal solution in the interval $[0,\beta]$. Moreover, $\overline u_\gamma \to \beta$ uniformly in $\overline\Omega$ when $\gamma\to 0$.
\end{proposition}

\begin{proof}
The existence of some $\overline\gamma>0$ for which a nontrivial solution to~\eqref{eq:PbAux2} exists for any $\gamma<\overline\gamma$ can be proved using variational methods in the same way as in Lemma~\ref{lem:exist_gam_fix}. Due to Proposition~\ref{prop:comparison} and Corollary~\ref{cor:comparison}, this solution is positive and has its maximum in $(\alpha,\beta)$. As $\beta$ is always a supersolution of~\eqref{eq:PbAux2}, for any $\gamma<\overline\gamma$ the maximal solution $\ou_\gamma$ to~\eqref{eq:PbAux2} in the interval $[0,\beta]$ exists.

Now, observe that $\ou_\gamma(x)$ is a decreasing function in $\gamma$ for any $x\in \overline\Omega$. Indeed, let $0<\gamma_1<\gamma_2<\overline\gamma$. Then $\ou_{\gamma_2}$ is a strict subsolution to~\eqref{eq:PbGamma} with $\gamma=\gamma_1$. Taking into account that $\ou_{\gamma_1}$ is the maximal solution of~\eqref{eq:PbGamma} with $\gamma=\gamma_1$ and using Proposition~\ref{prop:comparison}, we easily deduce that $\ou_{\gamma_2}< \ou_{\gamma_1}$ in $\overline\Omega$.

Therefore, we can define $u(x)=\lim_{\gamma\to 0} \ou_\gamma (x)$. As $\ou_\gamma$ is bounded in $\co$ by $\beta$, Lemma~\ref{lem:stability} implies that $\ou_\gamma\to u$ strongly in $\co$ and that $u\in C^2(\overline\Omega)$ is a solution of the Neumann problem
\begin{equation}
\label{eq:Pf_Lem_exist_2}
\begin{cases}
\displaystyle -\Delta u  = \tilde f(u) & \mbox{in} \; \Omega,\\
\frac{\partial u}{\partial \nu} = 0 & \mbox{on} \; \partial \Omega.
\end{cases} 
\end{equation}

Since $\|\ou_\gamma\|_{\co}\in (\alpha,\beta)$ for every $0<\gamma < \overline\gamma$, then $\|u\|_{\co} \in [\alpha,\beta]$, but notice that $\|u\|_{C(\overline\Omega)}\neq \alpha$ due to the monotonicity of $\ou_\gamma$. We stress that any solution to this Neumann problem must satisfy the compatibility condition
\[
\into \tilde f(u)=0.
\]
This condition is obtained by just taking the constant function $1$ as test function in the weak formulation of~\eqref{eq:Pf_Lem_exist_2}. As $\tilde f\geq 0$ and $u$ is continuous, then $\{u(x):x\in\Omega\}$ is a connected set that must be contained in some connected component of $\{s\geq 0: f(s)=0\}$. As $\|u\|_{C(\overline\Omega)}\in (\alpha,\beta]$ and $f(s)>0$ when $s\in (\alpha,\beta)$, then $\|u\|_{C(\overline\Omega)} = \beta$ and we can conclude that $u\equiv \beta$ in $\overline\Omega$.
\end{proof}

\section{Proofs of the main results}
\label{sec:main}

In this section, we prove the results stated in the introduction. To show Theorem~\ref{th:exist_gam_fix}, the strategy is to consider the truncated problem~\eqref{eq:PbAux} and use the uniform convergence of its solutions to $\beta$. Then, for sufficiently large $\lambda$, these solutions also solve~\eqref{eq:PbLambda}.

\begin{proof}[Proof of Theorem \ref{th:exist_gam_fix}]
Consider $\gamma>0$ fixed. Let $\overline v_{\lambda,\gamma}$ be the solution to~\eqref{eq:PbAux} given by Lemma~\ref{lem:exist_gam_fix} for $\lambda$ greater than a constant. By Proposition~\ref{lem:conv_gam_fix}, we know that $\overline v_{\lambda,\gamma}\to \beta$ in $C(\overline\Omega)$ as $\lambda\to\infty$, so there is some $\overline \lambda>0$ such that, for any $\lambda>\overline\lambda$, it holds $\overline v_{\lambda,\gamma}(x)\in (\alpha,\beta)$ for all $x\in \overline\Omega$. For later purposes, we take $\overline \lambda$ minimal, i.e., we take
\[
\overline\lambda(\gamma) =\inf \left\{\lambda>0 : \overline v_{\lambda,\gamma}(x)\in [\alpha,\beta) \text{ for all } x\in\overline\Omega \right\}.
\]
We point out that, since $\overline v_{\lambda,\gamma}$ is increasing and less than $\beta$, the set over which we take the infimum is an interval. In fact, this set is closed due to Lemma~\ref{lem:stability}.

As $\tilde f(s)=f(s)$ when $s\in(\alpha,\beta)$, we obtain that $\overline v_{\lambda,\gamma}$ is also a solution of~\eqref{eq:PbLambda} for any $\lambda \geq \overline\lambda$. We stress that $\overline v_{\lambda,\gamma}$ is also the maximal solution of~\eqref{eq:PbLambda} in $[0,\beta]$ when $\lambda\geq \overline\lambda$. To avoid mixing concepts, we denote by $\overline u_{\lambda,\gamma}$ the maximal solutions of~\eqref{eq:PbLambda} in $[0,\beta]$ whenever they exist. Of course, we have that $\ou_{\lambda,\gamma} \equiv \overline v_{\lambda,\gamma}$ for all $\lambda \geq \overline\lambda$.

Next, we want to show the multiplicity of nonnegative solutions for $\lambda>\overline\lambda$. Then, following the same notation as in the statement of Theorem~\ref{th:exist_gam_fix}, we can simply take $\lambda_\mathrm{mult}(\gamma)$ as $\overline \lambda (\gamma)$. Here, the key is that, since $\ou_{\overline\lambda, \gamma}(x)\in [\alpha,\beta)$ for any $x\in \overline\Omega$ and $f(s)>0$ when $s\in (\alpha, \beta)$, we can ensure that $\ou_{\overline\lambda, \gamma}$ is a strict subsolution of~\eqref{eq:PbLambda} for any $\lambda>\overline{\lambda}$.

From now on, we fix $\lambda > \overline\lambda$ and we consider the operator $K\colon \co \to C^2(\overline\Omega)$ defined in~\eqref{eq:def_K}. We recall that any fixed point of $K$ is a solution of~\eqref{eq:PbLambda} and vice versa. On the one hand, Lemma~\ref{lem:LS_1} ensures that $\operatorname{deg}(I-K, \mathcal{O}_{\alpha\beta}, 0)=0$. On the other hand, defining
\[
\mathcal{O}_1 = \left\{u\in\co : \ou_{\overline\lambda, \gamma} (x)<u(x)<\beta \text{ in } \overline\Omega\right\},
\]
and applying Lemma~\ref{lem:LS_2}, we obtain that $\operatorname{deg}(I-K, \mathcal{O}_1, 0)=1$. As $\mathcal{O}_1\subset \mathcal{O}_{\alpha\beta}$, the additivity property of the Leray-Schauder degree implies that
\begin{equation}
\label{eq:Pf_Th_exist}
\operatorname{deg}\left(I-K, \mathcal{O}_{\alpha\beta} \setminus \overline{\mathcal{O}_1}, 0\right) = -1.
\end{equation}

Since $\ou_{\lambda, \gamma}$ always belongs to $\mathcal{O}_1$ because of its maximality, thanks to~\eqref{eq:Pf_Th_exist} we can conclude that a second solution $\tilde u_{\lambda, \gamma}$ to~\eqref{eq:PbLambda} in $\mathcal{O}_{\alpha\beta}$ exists whenever $\lambda > \overline\lambda$.

It remains for us to show that $\overline\lambda(\gamma)$ (and then $\lambda_\mathrm{mult}(\gamma)$) converges to 0 as $\gamma\to 0$. The proof is complete if we prove that, for any $\lambda>0$, there exists some $\tilde\gamma>0$ small such that $\overline\lambda(\gamma) \leq \lambda$ for any $\gamma<\tilde\gamma$.

With this aim, we fix $\lambda>0$. Now, we use Proposition~\ref{prop:exist} to ensure the existence of some $\tilde \gamma>0$ depending on $\lambda$ such that, for any $\gamma<\tilde\gamma$, a maximal solution $\overline v_{\lambda,\gamma}$ to~\eqref{eq:PbAux2} (with $\tilde f$ replaced by $\lambda \tilde f$) exists verifying $\overline v_{\lambda,\gamma}(x) \in [\alpha,\beta)$ for all $x\in\overline\Omega$. Due to the definition of $\overline\lambda(\gamma)$, we conclude that $\overline\lambda(\gamma) \leq \lambda$ for all $\gamma<\tilde \gamma$.
\end{proof}

\begin{remark}
    If $f(0)\geq 0$, then any nonnegative (and nontrivial) solution to the Robin problem~\eqref{eq:PbLambda} is actually positive in $\overline\Omega$. This follows from Proposition~\ref{prop:comparison}, since in this case the constant zero function is a subsolution.
\end{remark}

Now, we can easily prove Theorem~\ref{th:lambda_bh} taking into account the well-known results of existence and nonexistence of solution for the Dirichlet problem~\eqref{eq:PbDirich}.

\begin{proof}[Proof of Theorem~\ref{th:lambda_bh}]
Let $\gamma>0$ be fixed. Consider the set
\[
A_\gamma:= \left\{\lambda\geq 0: \textrm{\eqref{eq:PbLambda} admits solution $u\geq 0$ with } \alpha < \|u\|_{C(\overline\Omega)}<\beta \right\}.
\]
Thanks to Theorem~\ref{th:exist_gam_fix}, we have that $A_\gamma \neq \emptyset$. By Lemma~\ref{lem:stability} and Corollary~\ref{cor:comparison}, we can also ensure that this set is closed. Then, we can define
\[
\lambda_\mathrm{min}(\gamma) = \min A_\gamma.
\]
Since $0\notin A_\gamma$, then $\lambda_\mathrm{min}(\gamma)>0$. By definition, we have that $\lambda_\mathrm{min}(\gamma) \leq \lambda_\mathrm{mult}(\gamma)$, where $\lambda_\mathrm{mult}(\gamma)>0$ is the one given in Theorem~\ref{th:exist_gam_fix}. Since $\lambda_\mathrm{mult}(\gamma)\to 0$ as $\gamma\to 0$, we immediately obtain that $\lambda_\mathrm{min}(\gamma)\to 0$ when $\gamma\to 0$.

Finally, we study the behaviour of $\lambda_\mathrm{min}(\gamma)$ as $\gamma$ goes to $\infty$.

On the one hand, if~\eqref{eq:hyp_hess} does not hold, then $\lambda_\mathrm{min}(\gamma)$ cannot have any accumulation point at infinity. The reason is that, if some accumulation point $\lambda$ exists, one could apply Lemma~\ref{lem:stability} to find a solution to the Dirichlet problem~\eqref{eq:PbDirich} belonging to $\mathcal{O}_{\alpha\beta}$ for this $\lambda$, but this is a contradiction with~\cite[Theorem~1]{Cl-Sw}. Therefore, we deduce that $\lambda_\mathrm{min}(\gamma)\to \infty$ as $\gamma\to\infty$.

On the other hand, if~\eqref{eq:hyp_hess} holds, then
\[
\lambda_\infty := \min \left\{\lambda\geq 0: \textrm{\eqref{eq:PbDirich} admits solution $u\geq 0$ with } \alpha < \|u\|_{C(\overline\Omega)}<\beta \right\}
\]
is finite and positive by~\cite[Theorem~1]{Cl-Sw}. Let $w_{\lambda_\infty}$ be a nonnegative solution of~\eqref{eq:PbDirich} for $\lambda=\lambda_\infty$ with maximum between $\alpha$ and $\beta$. Thanks to the Hopf's Lemma, $w_{\lambda_\infty}$ is a subsolution of $(P_{\lambda_\infty,\gamma})$ for any $\gamma>0$ and, as $\beta$ is always a supersolution, we deduce that $\lambda_\mathrm{min}(\gamma) \leq \lambda_\infty$ for all $\gamma>0$. Due to the definition of $\lambda_\infty$ (and thanks to Lemma~\ref{lem:stability}), the only possible accumulation point of $\lambda_\mathrm{min}(\gamma)$ at infinity is $\lambda_\infty$. Then, we can conclude that $\lambda_\mathrm{min}(\gamma)\to \lambda_\infty$ as $\gamma\to\infty$.
\end{proof}

\begin{remark}
In general, we are not able to show that problem~\eqref{eq:PbLambda} has two solutions in $\mathcal{O}_{\alpha\beta}$ for any $\lambda>\lambda_\mathrm{min}$. However, when $f\geq 0$, this can be easily proved following the same reasoning as in Theorem~\ref{th:exist_gam_fix}. Specifically, one must take into account that the maximal solution to $(P_{\lambda_\mathrm{min}, \gamma})$ is a strict subsolution of ~\eqref{eq:PbLambda} for any $\lambda>\lambda_\mathrm{min}$.

Even more, when $f\geq 0$ one can follow~\cite[Theorem~1.1]{C-MA-MA} to show that, for any fixed $\gamma>0$, there is an unbounded continuum of solutions
\[\Sigma\subset \{(\lambda,u)\in \R_0^+\times \mathcal{O}_{\alpha\beta}: u \text{ solves } \eqref{eq:PbLambda}\}\]
with $\subset$-shape such that $\lambda_\mathrm{min} = \min(\operatorname{Proj_\lambda} \Sigma)$ and that, for every $\lambda>\lambda_\mathrm{min}$, the set $(\{\lambda\}\times \mathcal{O}_{\alpha\beta})\cap \Sigma$ has at least two elements.
\end{remark}

To conclude this section, we prove Theorem~\ref{th:limit_gamma}, which provides some information about the behaviour of the two solutions found in Theorem~\ref{th:exist_gam_fix} when $\gamma$ approaches zero and $\lambda$ remains fixed.

\begin{proof}[Proof of Theorem~\ref{th:limit_gamma}]

    By Theorem~\ref{th:exist_gam_fix}, we know that $\lambda_{\mathrm{mult}}(\gamma)\to 0$ when $\gamma\to 0$. Consequently, we can find $\tilde \gamma>0$ such that $\lambda_\mathrm{mult}(\gamma)<1$ for all $\gamma<\tilde\gamma$. Then, for $\gamma< \tilde \gamma$, problem~\eqref{eq:PbGamma} has two solutions in $\mathcal{O}_{\alpha\beta}$, which we denote by $\ou_\gamma$ and $\tilde u_\gamma$, where $\ou_\gamma$ is the maximal solution in $[0,\beta]$.

    Arguing as in Proposition~\ref{prop:exist}, we can prove that $\ou_\gamma(x)$ is a nonincreasing function in $\gamma$ for any $x\in \overline\Omega$ and that $\ou_\gamma \to \beta$ in $\co$ when $\gamma\to 0$.

    On the other hand, Lemma~\ref{lem:stability} implies that $\tilde u_\gamma\to \tilde u$ in $\co$, where $\tilde u$ denotes a solution for the Neumann problem~\eqref{eq:PbNeumann}. Notice that $\|\tilde u\|_\co\in [\alpha,\beta]$. Since we are assuming that~\eqref{eq:PbNeumann} has no non-constant solutions with maximum in $[\alpha,\beta]$, it follows that either $\tu\equiv \alpha$ or $\tu\equiv \beta$. Our aim is to prove that only the first case can happen. To do that, we follow the ideas of~\cite{Cl-Sw}. The strategy is to show that~\eqref{eq:PbGamma} admits at most one solution in a left neighbourhood of $\beta$ when $\gamma$ is small enough.

    First, we prove that any solution $u$ to~\eqref{eq:PbGamma} contained in $(\beta-\delta,\beta)$ is stable, where $\delta$ is given by~\eqref{eq:hyp_mon}. Recall that a solution $u$ to~\eqref{eq:PbGamma} is stable if the first eigenvalue $\mu_1(\gamma, f'(u))$ of the problem
    \begin{equation*}
    \begin{cases}
    \displaystyle -\Delta v - f'(u)v = \mu v & \mbox{in} \; \Omega,\\
    \frac{\partial v}{\partial \nu} + \gamma v = 0 & \mbox{on} \; \partial \Omega,
    \end{cases} 
    \end{equation*}
    is positive, i.e., if $\mu_1(\gamma, f'(u))>0$. This eigenvalue is given by
    \[
    \mu_1(\gamma, f'(u)) = \inf_{\varphi\in\hu, \|\varphi\|_{L^2(\Omega)=1}} \left( \into \nabla \varphi^2 + \gamma \intdo |\varphi|^2  - \into f'(u)\varphi^2 \right).
    \]
    If $u$ is such that $u(x)\in(\beta-\delta,\beta)$ for every $x\in\overline\Omega$, then hypothesis~\eqref{eq:hyp_mon} implies that $f'(u)\leq 0$ in $\Omega$. In this way, we deduce that $\mu_1(\gamma,f'(u)) \geq \mu_1(\gamma,0)>0$, where we have taken into account that $\mu_1(\gamma,0)$ is just the first eigenvalue of the Robin Laplacian, which is positive. Therefore, $u$ is stable. 

    Let $\gamma_0>0$ be such that $\ou_{\gamma_0}(x)\in (\beta-\delta,\beta)$ for all $x\in \overline\Omega$. We claim that~\eqref{eq:PbGamma} has at most one solution contained in the set
    \[
    \mathcal{O}_1 = \{u\in\co : \ou_{\gamma_0}(x)<u(x)<\beta \text{ in } \overline\Omega\}
    \]
    when $\gamma<\gamma_0$. In fact, for any $\gamma<\gamma_0$ function $\ou_{\gamma_0}$ is a strict subsolution of~\eqref{eq:PbGamma} and $\beta$ is a strict supersolution, so Lemma~\ref{lem:LS_2} implies that
    \[
    \operatorname{deg} (I-K, \mathcal{O}_1,0) = 1,
    \]
    where $K$ is defined in~\eqref{eq:def_K} (with $\lambda = 1$). 
    
    Since any solution to~\eqref{eq:PbGamma} contained in $(\beta-\delta,\beta)$ is stable, then any such solution has index 1. This is because any stable solution is an isolated fixed point of $K$, which implies that its index is 1 (see~\cite{Rab2}). Together with the the additivity of the topological degree, this allows us to deduce that there is only one solution to~\eqref{eq:PbGamma} contained in $\mathcal{O}_1$.

    As $\ou_\gamma \in \mathcal{O}_1$ for every $\gamma<\gamma_0$, the sequence $\tu_\gamma$ cannot converge to $\beta$ in $\co$. Therefore, we conclude that $\tu\to \alpha$ in $\co$.
\end{proof}

\section*{Acknowledgements}
This work has been funded by the Spanish Ministry of Science and Innovation, the Agencia Estatal de Investigación (AEI), and the European Regional Development Fund (ERDF), under grant PID2021-122122NB-I00. The first and second authors have also been supported by the Junta de Andalucía under grant FQM-194, and the third author under grant FQM-116. Additionally, the second author has been supported by the FPU predoctoral fellowship of the Spanish Ministry of Universities (FPU21/04849). The first and third authors have also been funded by CDTIME.
	


\end{document}